\documentclass[a4paper,french, english]{article}

\usepackage[utf8]{inputenc}
\usepackage[T1]{fontenc}
\usepackage{babel}
\usepackage[babel]{csquotes}

\usepackage{amsmath}
\usepackage{amsfonts}
\usepackage{amsthm}
\usepackage{bbm}
\usepackage{amssymb}
\usepackage[top=3cm, bottom=3cm, left=3.3 cm, right=3.3 cm]{geometry}
\usepackage{mathrsfs}
\usepackage{hyperref}
\usepackage{stmaryrd}

\usepackage{float}
\usepackage{tikz}
\usetikzlibrary{shapes.misc}
\tikzset{cross/.style={cross out, draw, 
         minimum size=2*(#1-\pgflinewidth), 
         inner sep=0pt, outer sep=0pt}}
\usepackage{IEEEtrantools}
\usepackage{url}

\newcommand{\comm}[1]{{\color{red}#1}}

\newtheorem{theorem}{Theorem}

\newtheorem{lemma}[theorem]{Lemma}
\newtheorem{corollary}[theorem]{Corollary}

\newtheorem*{theorem*}{Theorem}
\theoremstyle{definition}
\newtheorem*{definition*}{Definition}

\DeclareMathOperator{\Ad}{Ad}

\newcommand{\Z}{\mathbb{Z}}
\newcommand{\N}{\mathbb{N}}

\newcommand{\PP}{\mathbb{P}}

\newcommand{\supp}{\text{supp}\,}

\newcommand{\abs}[1]{\lvert#1\rvert}    
\newcommand{\norm}[1]{\lVert#1\rVert}   

\title{Equidistribution of mass for random processes on finite-volume spaces}
\author{Timothée Bénard}
\date{}
\begin{document}

\maketitle

\bigskip

\begin{abstract}
Let $G$ be a real Lie group, $\Lambda\subseteq G$ a lattice,  and $X=G/\Lambda$. We fix a probability measure $\mu$ on $G$ and consider the left random walk induced on $X$. It is assumed that $\mu$ is aperiodic, has a finite first moment, spans a  semisimple algebraic group without compact factors, and has two non mutually singular convolution powers. We show that for every starting point $x\in X$, the $n$-th step distribution $\mu^n*\delta_{x}$ of the walk weak-$\ast$ converges toward some homogeneous probability measure on $X$. 
\end{abstract}
\large

\bigskip
\section{Introduction}
In this note we consider a semigroup $\Gamma$ acting on a space $X$ and study the way the $\Gamma$-orbits distribute in $X$. A first approach might be to describe the orbit closures. In the case where $\Gamma$ is a diagonal group acting on a finite-volume homogeneous space, it is known that these orbit closures may be extremely irregular, with transverse sections homeomorphic to a Cantor set. However, Ratner's theorems \cite{Rat} highlight that on the opposite, the orbits of a group generated by Ad-unipotent 1-parameter flows are very regular. Benoist-Quint \cite{BQIII} recently proved that such regularity also occurs for the action of a semigroup $\Gamma$ whose Zariski closure is semisimple. Our goal is to complement Benoist-Quint's result by estimating how $\Gamma$-orbits equidistribute in their closure from the perspective of random walks supported on $\Gamma$. We first recall the state of the art concerning Benoist-Quint's theorems.

Let $G$ be a real Lie group, $\Lambda\subseteq G$ a lattice, and  set $X=G/\Lambda$. We consider the action by left multiplication on $X$ of a closed subsemigroup $\Gamma\subseteq G$. Thoughout the text, $\Gamma$ will always be assumed to satisfy  

\begin{itemize}
\item[$(\mathbf{H})$] The  real algebraic group  $\overline{\Ad\Gamma}^Z\subseteq \text{Aut}(\mathfrak{\mathfrak{g}})$ generated by $\Gamma$ in the adjoint representation  is semisimple, Zariski connected, with no compact factor. 
\end{itemize}

We also need a precise definition to express the regularity of orbit closures.  

\begin{definition*}
 A closed subset $Y$ of $\Omega$ is \emph{homogeneous} if its stabilizer $G_{Y}=\{g\in G, \,gY=Y\}$ acts transitively on $Y$. We add that $Y$ has \emph{finite volume} if the action of $G_{Y}$ on $Y$ preserves a Borel probability measure on $Y$. Such a measure is then unique, denoted by $\nu_{Y}$. If the semigroup $\Gamma$ is included in $G_{Y}$ (and acts ergodically on $(Y,\nu_{Y})$), we say that  $Y$ is \emph{$\Gamma$-invariant} (and $\Gamma$-\emph{ergodic}). 

\end{definition*}

\let\thefootnote\relax\footnotetext{The  author has received funding from the European Research
Council (ERC) under the European Union’s Horizon 2020 research and
innovation programme (grant agreement No. 803711).}

\newpage

The  following result  is essentially due to Benoist-Quint.

\begin{theorem*}[\cite{BQIII, BenDeS}]

For every $x\in X$, the orbit closure $\overline{\Gamma.x}$ is  a  finite-volume homogeneous closed subset of $X$ which is $\Gamma$-invariant ergodic. 
\end{theorem*}

The proof also yields that random walks on $\Gamma$ equidistribute in the $\Gamma$-orbits in Cesàro average.  This is due to Benoist-Quint for walks with bounded jumps, and to Bénard-De Saxcé for walks with  finite first moment. 

\begin{theorem*}[\cite{BQIII, BenDeS}]
Let $\mu$ be a Borel probability measure on $\Gamma$ with a finite first moment and whose support spans a dense subsemigroup of $\Gamma$. 

Then for every $x\in X$, we have the weak-$\ast$ convergence
\begin{align}\label{equid}
\frac{1}{n}\sum_{k=0}^{n-1}\mu^k\ast \delta_{x} \,\underset{n\to +\infty}{\longrightarrow}\, \nu_{x}
\end{align}
where $\nu_{x}:=\nu_{\overline{\Gamma.x}}$ denotes the homogeneous probability measure on  $\overline{\Gamma.x}$. 
\end{theorem*}

In this statement, $\mu^k$ denotes the $k$-fold convolution of $\mu$ and the assumption of finite first moment on $\mu$ means that  $$\int_{G}\log \,\norm{\text{Ad} g}\,d\mu(g)<\infty$$
Moreover, the weak-$\ast$ topology refers to the pointwise convergence topology for the space of finite Borel measures on $X$ seen as linear forms on the space of continuous bounded functions on $X$. 

Benoist-Quint  ask in the $10^{\text{th}}$ Takagi Lectures \cite[p. 29]{BQTakagi} whether the  equidistribution result (\ref{equid}) still holds without Cesàro averages. Our goal is to answer positively under the condition that $\mu$  has two distinct powers which are not mutually singular. Note also that an aperiodicity condition is required  to avoid situations where $\mu^{*n}*\delta_{x}$ is obviously non converging; for example, the case where there exists a group $\Lambda'$ containing $\Lambda$ as a finite-index normal subgroup and such that $\mu$ is supported on a non trivial class in $\Lambda'/\Lambda$ \cite{Proh19}.  We will say that $\mu$ is \emph{aperiodic} if its support is not contained in  a fibre of a non trivial continuous morphism of semigroups from $\Gamma$ to a finite (cyclic) group.  

\begin{theorem}[Equidistribution of mass] \label{TH}
Let $G$ be a real Lie group, $\Lambda\subseteq G$ a lattice, and $\Gamma\subseteq G$ a closed subsemigroup satisfying $(\mathbf{H})$. Let $\mu$ be an aperiodic Borel probability measure on $\Gamma$ with a finite first moment and whose support spans a dense subsemigroup of $\Gamma$. Assume there exist $k\neq l\in \N$ such that $\mu^{k}$ is not singular with $\mu^{l}$. 

Then for every $x\in X$, we have the weak-$\ast$ convergence
$$\mu^n\ast \delta_{x} \underset{n\to +\infty}{\longrightarrow}  \nu_{x}$$
where $\nu_{x}:=\nu_{\overline{\Gamma.x}}$ denotes the homogeneous probability measure on  $\overline{\Gamma.x}$. 
\end{theorem}

In particular, we obtain equidistribution of mass  when $\mu$ is \emph{symmetric with countable support}, or if $\mu$ is aperiodic with $\mu^k(e)>0$ for some $k\geq 1$.

\section{Proof of Theorem \ref{TH}}

Given a signed measure $m$ on a measurable space, we denote by $\norm{m}$ its total variation
$$\norm{m}=\sup \sum_{i\in \N} |m(A_{i})|  $$ 
 where the supremum is taken over all the measurable countable partitions $(A_{i})_{i\in \N}$ of the space.

\begin{lemma}\label{dif}
Let $H$ be a measurable group and $\mu$ a  probability measure on $H$ such that $\mu^k, \mu^l$ are not mutually-singular for some distinct integers $ k\neq l\in \N$. Then we have the convergence
$$\norm{\mu^n-\mu^{n+k-l} }\underset{n\to +\infty}{\longrightarrow} 0 $$
\end{lemma}

\begin{proof}
This result is due to Foguel \cite[Theorem 1]{Fog75}, whose proof is strongly inspired by the proof of \cite[Theorem 1.1]{OrnSuch70} about conservative ergodic Markov operators. 
The argument is a bit tricky. For completeness, we give a more intuitive explanation in the case where $l=0$, i.e. $\mu^k(e)>0$ for some $k\geq 1$. 

We may assume $k=1$. 
Decompose $\mu$ as $\mu=\alpha \delta_{e}+(1-\alpha)\mu_{0}$ where $\alpha>0$ and $\mu_{0}$ is a probability measure on $H$. Then for $n\geq0$,
$$\mu^n= (\alpha \delta_{e}+(1-\alpha)\mu_{0})^n=\sum_{i=0}^n \binom{n}{i} \alpha^{n-i}(1-\alpha)^i\mu^{i}_{0}  $$
Hence, using the convention $ \binom{n}{i}=0$ if $i\geq n+1$,   
\begin{align*}
\norm{\mu^n-\mu^{n+1}}&= \norm{\sum_{i=0}^{n+1} \left(\binom{n}{i} \alpha^{n-i}(1-\alpha)^i  -\binom{n+1}{i} \alpha^{n+1-i}(1-\alpha)^i \right)\mu^{i}_{0}} \\
&\leq \sum_{i=0}^{n+1} \,\abs{\binom{n}{i} \alpha^{n-i}(1-\alpha)^i  -\binom{n+1}{i} \alpha^{n+1-i}(1-\alpha)^i }\\
&= \sum_{i=0}^{n+1}\, \abs{\PP(S_{n}=i)  -\PP(S_{n+1}=i)  }
\end{align*}
where $S_{n}$ denotes the $n$-th step  of the Markov chain on $\N$ starting at the origin and with i.i.d. increments of law $\alpha\delta_{0} +(1-\alpha) \delta_{1}$. A straightforward computation shows that 
$$\PP(S_{n}=i) \geq \PP(S_{n+1}=i)  \iff i \leq (n+1)(1-\alpha)=:m_{n}$$
Hence, 
\begin{align*}
\norm{\mu^n-\mu^{n+1}}\leq \abs{\PP(S_{n}\leq m_{n}) - \PP(S_{n}> m_{n})}+ \abs{\PP(S_{n+1}\leq m_{n}) - \PP(S_{n+1}> m_{n})}
\end{align*}

By the  Central Limit Theorem, each term in the right hand side converges to zero as $n$ goes to infinity, whence the result.  
\end{proof}

\bigskip
Lemma \ref{dif} has the following remarkable corollary. Recall that given an action of $H$ on a  space $X$, a  measure $\nu$ on $X$ is said to be \emph{$\mu$-stationary} if $\nu=\int_{H}g_{\ast}\nu \,d\mu(g)$. 

\begin{corollary} \label{stat} Let $H$ be a topological group acting on a topological space $X$, and $\mu$ a Borel probability measure on $H$ such that $\mu^k, \mu^l$ are not mutually-singular for some distinct integers $ k>l\geq 0$.

 Then for every $x\in X$, all the weak-$\ast$ limits of the sequence $(\mu^{(k-l)n}\ast \delta_{x})_{n\geq 0}$ are $\mu^{k-l}$-stationary. 
\end{corollary}

\noindent{\bf Remark}. The weak-$\ast$ limits in question are not assumed to be of mass $1$, there could be the null measure for instance.

\bigskip

 \begin{proof}[Proof of Theorem \ref{TH}]
 Set $d=|k-l|\geq 1$ then let $\Gamma_{d}$ be the closed semigroup generated by the support of $\mu^d$ and $\Gamma^d=\{g^d, g\in \Gamma\}$. The observations that $\Gamma_{d}\supseteq \Gamma^d$ and  $(\text{Ad}\Gamma)^d$ is Zariski dense in $\text{Ad}\Gamma$ lead to the equality $\overline{\text{Ad}\Gamma_{d}}^Z =\overline{\text{Ad}\Gamma}^Z$.  In particular, we may apply Benoist-Quint's theorems presented above to $\mu^d$. Fix $x\in X$ and  denote by $\nu_{d,x}$ the homogeneous probability measure on $\overline{\Gamma_{d}.x}$.   
 
 We first prove the weak-$\ast$ convergence 
  \begin{align}\label{conv-d}
  \mu^{nd}\ast \delta_{x}\,\longrightarrow \,\nu_{d,x}
  \end{align}
Let $\nu$ be a weak-$\ast$ limit of $(\mu^{nd}\ast \delta_{x})_{n\geq 0}$. It is enough to show that $\nu=\nu_{d,x}$. We know from  \cite[Theorem A]{BenDeS} and Corollary \ref{stat} that $\nu$ is a $\mu^d$-stationary probability measure on $\overline{\Gamma_{d}.x}$. Moreover, for every $y\in \overline{\Gamma_{d}.x}$, we have by \cite[Theorem C]{BenDeS} that $\frac{1}{n}\sum_{k=0}^{n-1}\mu^{kd}\ast \delta_{y}$ converges to the homogeneous measure on $\overline{\Gamma_{d}.y}$. Hence, we just need to check that $\nu$ is concentrated on the set of $y\in \overline{\Gamma_{d}.x}$ such that $\overline{\Gamma_{d}.y}=\overline{\Gamma_{d}.x}$. But, by  \cite[Lemma 6.1]{BenDeS}, the complementary subset  is included in some countable union $\cup_{i\in \N}LY_{i}$ where $L=Z_{G}(\Gamma_{d})$ denotes the centralizer of $\Gamma_{d}$ in $G$, and each $Y_{i}\subsetneq \overline{\Gamma_{d}.x}$ is a proper $\Gamma_{d}$-invariant ergodic finite-volume homogeneous subspace. Noticing that $x\notin L Y_{i}$ and applying the result of non accumulation of mass established in \cite[Theorem B']{BenDeS}, we get that $\nu(LY_{i})=0$ for each $i\in \N$, which finishes the proof of (\ref{conv-d}).

    Now combining (\ref{conv-d}) with the fact that $\mu^n\ast\delta_{x}$ converges toward $\nu_{x}$ in Cesàro averages,  we obtain
 $$\nu_{x}=\frac{1}{d}\left(\nu_{d,x}+\mu\ast \nu_{d,x}+\dots+\mu^{d-1}\ast \nu_{d,x}   \right)$$
  To conclude, it is enough to show that $\nu_{d,x}$ is $\Gamma$-invariant.  This is where we use the aperiodicity assumption.  We denote by $Y_{i}$ the support of  $\mu^i\ast \nu_{d,x}$ and check that every $g\in \Gamma$ acts as a permutation of the set $\{Y_{i}, i\in \Z/d\Z\}$. We can assume that $g\in \supp \mu$. The observation that $Y_{k+i}=\overline{(\supp \mu)^k Y_{i}}$ implies in particular that $gY_{i}\subseteq Y_{i+1}$.   Then, $Y_{i}$ being $\Gamma_d$-invariant, we have  $Y_{i}=g^dY_{i}\subseteq g^{d-1}Y_{i+1} \subseteq Y_{i}$, whence $gY_{i}=Y_{i+1}$. Finally, $\Gamma$ does permute the components  $\{Y_{i}, i\in \Z/d\Z\}$, and the  elements of $\supp \mu$ all act as the same permutation. The aperiodicity condition yields that the $Y_{i}$'s are all equal, or in other words that $\Gamma$ stabilizes the support of $\nu_{d,x}$. By homogeneity, this means that $\nu_{d,x}$ is $\Gamma$-invariant and concludes the proof.  

 \end{proof}

 \newpage
\nocite{el_rw}
\bibliographystyle{abbrv}

\bibliography{bibliographie}

\end{document}